\documentclass[preprint,12pt]{elsarticle}

\usepackage{amsmath,amssymb,amsthm}
\usepackage{enumitem}
\usepackage{geometry}
\usepackage{hyperref}
\usepackage[utf8]{inputenc}
\usepackage{mathtools}
\newtheorem{theorem}{Theorem}[section]
\newtheorem{lemma}[theorem]{Lemma}
\newtheorem{proposition}[theorem]{Proposition}

\newtheorem{definition}[theorem]{Definition}
\newtheorem{remark}[theorem]{Remark}
\newtheorem{example}[theorem]{Example}

\begin{document}

\begin{frontmatter}

\title{On Quasi-Modular Pseudometric Spaces and Asymmetric Uniformities}

\author{Philani Rodney Majozi}
\ead{Philani.Majozi@nwu.ac.za}
\address{Department of Mathematics, Pure and Applied Analytics, North-West University, Mahikeng, South Africa}

\begin{abstract}
We study quasi-modular pseudometric spaces as asymmetric refinements of modular
metric structures. To each such space we associate canonical forward and backward
quasi-uniformities and the corresponding directional topologies.

We introduce directional notions of convergence, completeness, total boundedness,
and compactness, and show that these properties are not preserved under
symmetrization. In particular, forward and backward completeness may differ, and
compactness of the symmetrized uniformity does not imply directional compactness.

Using enriched category theory as a comparison framework, we show that
symmetrization yields a symmetric enriched category whose Cauchy completion
coincides with the classical uniform completion, while directional notions
remain invisible at this level.
\end{abstract}

\begin{keyword}
Quasi-modular pseudometric \sep asymmetric uniformity \sep quasi-uniform space \sep
Lawvere-enriched category \sep modular topology \sep completeness \sep compactness
\\[2pt]
\end{keyword}

\end{frontmatter}

\section{Introduction}\label{sec:intro}

Modular metric spaces, introduced by Chistyakov
\cite{Chistyakov2010,Chistyakov2015}, provide a flexible setting interpolating
between classical metric geometry and modular methods arising in functional
analysis. In this setting, displacement between two points $x,y$ is measured not
by a single distance, but by a scale-dependent family
$w_\lambda(x,y)$ indexed by $\lambda>0$. This scale parameter encodes resolution
or tolerance and allows one to recover, as particular instances, ordinary
metrics as well as Luxemburg-type distances associated with Orlicz and
Musielak-Orlicz spaces. The induced \emph{modular topology} $\tau(w)$, its
metrizability in convex cases, and its relationship with completeness have been
studied in detail in recent work by Mushaandja and Olela-Otafudu
\cite{MushaandjaOlela2025}.

The systematic study of \emph{asymmetric} modular structures predates the present
work. In particular, quasi-modular (or modular quasi-metric) spaces were
introduced and investigated by K.~Sebogodi in his Master's and Ph.D.\ studies under
the supervision of O.~Olela-Otafudu \cite{Sebogodi2019}, where several basic
topological properties and convergence notions were already identified. More
recently, López-Pastor, Pedraza, and Rodríguez-López
\cite{LopezPastor2025} showed that quasi-pseudometric modular spaces can be
naturally viewed as categories enriched over suitable quantales, thereby placing
the theory within the general setting of Lawvere-enriched category theory.

The present paper does not seek to reintroduce these structures nor to replicate
their foundational properties. Instead, it focuses on the genuinely
\emph{directional} notion that arise once symmetry is abandoned, and that are
not captured by either symmetric modular theory or by purely categorical
representations. Our guiding principle is that asymmetry encodes intrinsic
directionality: the ``effort'' to move from $x$ to $y$ may differ from the effort
to return. Such behavior is inherent in directed graphs, transport problems with
drift, causal or temporal precedence relations, irreversible dynamics, and
cost-to-go functions in control theory. From a topological point of view, these
features are naturally modeled by quasi-metrics and quasi-uniformities
\cite{Kunzi2001,Matthews1992,PicadoPultr2011}, whose associated bitopological
structures were systematically developed by Fletcher and Lindgren
\cite{Fletcher2014}.

In this work we develop a theory of \emph{quasi-modular pseudometric spaces} in
which the modular triangle inequality and scale regularity are retained, while
symmetry is dropped. Each such space canonically generates forward and backward
quasi-uniformities, together with corresponding left and right topologies.
Although these structures admit a natural symmetrization, yielding an ordinary
modular space and a classical uniformity, we show that many fundamental properties
are \emph{lost} under this transition. In particular, notions of convergence,
Cauchy completeness, total boundedness, and compactness exhibit genuinely
directional behavior that cannot be recovered from the symmetrized modular or
from the associated enriched category.

The categorical point of view nevertheless plays an important organizing role.
Following Lawvere’s interpretation of (quasi-)metrics as categories enriched
over ordered monoidal bases \cite{Lawvere1973}, and Kelly’s general theory of
enriched categories and Cauchy completion \cite{Kelly1982}, we show that every
quasi-modular gauge determines an enriched category over a convolution quantale.
Within this setting, symmetrization appears as a lax idempotent change-of-base
construction, and the enriched Cauchy completion coincides with the classical
uniform completion of the symmetrized space. While this perspective unifies and
clarifies several constructions, it also highlights a key limitation: enriched
and uniform completions capture only the symmetric core of the theory, leaving
directional completeness and compactness invisible.

The main contribution of this paper is therefore conceptual rather than
foundational. We identify and isolate those properties of quasi-modular
pseudometric spaces that are intrinsically asymmetric, prove that they are not
invariants of symmetrization, and show that they refine the classical uniform
notions. In particular, we introduce directional convergence and completeness,
establish strict separations between forward, backward, and symmetric
completeness, and demonstrate that compactness of the symmetrized uniformity does
not imply any directional compactness. These results show that quasi-modular
pseudometric spaces form a genuinely richer category than their symmetric or
categorical shadows.

The paper is organized as follows. Section~\ref{sec:preliminaries} recalls the
classical modular setting and introduces the quasi-modular relaxation.
Section~\ref{sec:asymmetry} develops the associated forward and backward
quasi-uniformities and the induced bitopological structure. In
Section~\ref{sec:uniform-categorical} we relate quasi-modular gauges to enriched
categories and compare enriched and uniform completions. Finally,
Section~\ref{sec:directional-notion} is devoted to directional convergence,
completeness, and compactness, and to the failure of these properties to be
preserved under symmetrization.

\section{Preliminaries and Motivation}\label{sec:preliminaries}

We recall the standard (symmetric) modular setting after Chistyakov
\cite{Chistyakov2010,Chistyakov2015} and fix notation. We then state the
asymmetric relaxation that underlies the paper and record elementary
consequences needed later (monotonicity in the scale, regularizations,
Luxemburg-type gauges, and basic examples). Classical modulars in the sense of
Nakano-Orlicz-Musielak \cite{Musielak1983,RaoRen1991} and uniform/metric
basics \cite{Isbell1964,Kunzi2001,Matthews1992} serve as background.

\subsection{Modulars, pseudomodulars, and convexity}

Let $X$ be a nonempty set. A mapping
\[
w:(0,\infty)\times X\times X\longrightarrow [0,\infty],\qquad
(\lambda,x,y)\mapsto w_\lambda(x,y),
\]
is a \emph{(metric) modular} on $X$ if, for all $\lambda,\mu>0$ and $x,y,z\in X$:
\begin{enumerate}[label=(M\arabic*)]
\item\label{M1} $x=y$ iff $w_\lambda(x,y)=0$ for all $\lambda>0$ \emph{(separation)};
\item\label{M2} $w_\lambda(x,y)=w_\lambda(y,x)$ \emph{(symmetry)};
\item\label{M3} $w_{\lambda+\mu}(x,y)\le w_\lambda(x,z)+w_\mu(z,y)$ \emph{(modular triangle)}.
\end{enumerate}
If \ref{M1} is relaxed to $w_\lambda(x,x)=0$ (all $x,\lambda$), we speak of a
\emph{pseudomodular}. A modular (or pseudomodular) is called \emph{convex} if
\begin{equation}\label{eq:convex}
w_{\lambda+\mu}(x,y)\le \frac{\lambda}{\lambda+\mu}w_\lambda(x,y)
+\frac{\mu}{\lambda+\mu}w_\mu(x,y)
\qquad(\lambda,\mu>0).
\end{equation}
Equivalently, $\lambda\mapsto \lambda\,w_\lambda(x,y)$ is nonincreasing
\cite[Ch.~1]{Chistyakov2015}. For fixed $x,y\in X$, we write
\[
w^{x,y}:(0,\infty)\to[0,\infty], \qquad
w^{x,y}(\lambda):=w_\lambda(x,y),
\]
for the scale profile associated with the pair $(x,y)$. In particular, for $0<\lambda\le\mu$,
\begin{equation}\label{eq:scale-monotone}
w_\mu(x,y)\le \frac{\mu}{\lambda}w_\lambda(x,y)\le w_\lambda(x,y),
\end{equation}
and whenever $w^{x,y}\not\equiv 0$ one has
$\lim_{\lambda\downarrow 0}w_\lambda(x,y)=\infty$, while
$\lim_{\lambda\uparrow\infty}w_\lambda(x,y)=0$ if $w^{x,y}\not\equiv\infty$
\cite[1.2.3]{Chistyakov2015}.

\paragraph{Right/left regularizations.}
Define
\[
(w_{+0})_\lambda(x,y):=\lim_{\mu\downarrow\lambda}w_\mu(x,y),\qquad
(w_{-0})_\lambda(x,y):=\lim_{\mu\uparrow\lambda}w_\mu(x,y).
\]
If $w$ is a (pseudo)modular (strict/convex, resp.), then $w_{\pm0}$ is again a
(pseudo)modular with the same additional properties; moreover $w_{+0}$ is
right-continuous and $w_{-0}$ is left-continuous in $\lambda$
\cite[1.2.5]{Chistyakov2015}. In what follows we tacitly replace $w$ by a
right-continuous representative $w_{+0}$ when convenient.

\subsection{From classical modulars and metrics to metric modulars}

Two basic templates produce metric modulars:

\begin{example}[{\cite[1.3.1--1.3.2]{Chistyakov2015}}]\label{ex:scaled-metric}
Let $(X,d)$ be a (pseudo)metric space and $g:(0,\infty)\to[0,\infty]$.
Then $w_\lambda(x,y):=g(\lambda)\,d(x,y)$ is a (pseudo)modular iff $g$ is
nonincreasing; it is strict iff $g(\lambda)\ne0$ for all $\lambda>0$, and convex
iff $\lambda\mapsto\lambda g(\lambda)$ is nonincreasing. The canonical choice
$g(\lambda)=\lambda^{-1}$ yields $w_\lambda=d/\lambda$.
\end{example}

\begin{example}[{\cite[1.3.4--1.3.5]{Chistyakov2015}}]\label{ex:classical}
Let $X$ be a real linear space and $\rho:X\to[0,\infty]$ a (convex) classical
modular. Then
\[
w_\lambda(x,y):=\rho\!\left(\frac{x-y}{\lambda}\right)
\]
is a (convex) metric modular. Conversely, if $w$ is translation invariant and
positively homogeneous in the sense of \cite[1.3.5]{Chistyakov2015}, then
$w_\lambda(x,y)=\rho((x-y)/\lambda)$ with $\rho(\cdot)=w_1(\cdot,0)$.
\end{example}

\subsection{The modular topology and uniformity (symmetric case)}

Given a (pseudo)modular $w$ on $X$, set
\[
B_w(x,\lambda,\varepsilon):=\{y\in X:\ w_\lambda(x,y)<\varepsilon\}
\qquad(\lambda,\varepsilon>0).
\]
These form a neighborhood base of a topology $\tau(w)$ on $X$ (Hausdorff iff $w$
is strict) and generate a uniformity via entourages
$V(\lambda,\varepsilon):=\{(x,y):w_\lambda(x,y)<\varepsilon\}$; the induced
uniform topology coincides with $\tau(w)$. In the convex case $\tau(w)$ is
metrizable by a Luxemburg-type distance
\cite[Chs.~2--4]{Chistyakov2015}; compare \cite{Isbell1964}.

\subsection{Quasi-modular pseudometrics: the asymmetric relaxation}

Our target objects arise by removing symmetry \ref{M2} and keeping the modular
triangle and scale monotonicity.

\begin{definition}
A family $w=\{w_\lambda\}_{\lambda>0}$ with $w_\lambda:X\times X\to[0,\infty]$
is a \emph{quasi-modular pseudometric} if:
\begin{enumerate}[label=(QM\arabic*)]
\item $w_\lambda(x,x)=0$ for all $x\in X$, $\lambda>0$;
\item $w_{\lambda+\mu}(x,z)\le w_\lambda(x,y)+w_\mu(y,z)$ for all $x,y,z$ and $\lambda,\mu>0$;
\item $\lambda\mapsto w_\lambda(x,y)$ is right-continuous and nonincreasing.
\end{enumerate}
If, additionally, $w_\lambda(x,y)=0=w_\lambda(y,x)$ for all $\lambda>0$ implies
$x=y$, we say $w$ is \emph{$T_0$-separating}.
\end{definition}

\paragraph{Forward and backward structures.}
Associated with a quasi-modular pseudometric $w$ are the forward and backward
basic entourages
\[
V^+(\lambda,\varepsilon)=\{(x,y):w_\lambda(x,y)<\varepsilon\},\qquad
V^-(\lambda,\varepsilon)=\{(x,y):w_\lambda(y,x)<\varepsilon\},
\]
which generate, respectively, the forward and backward quasi-uniformities.

\begin{definition}\label{def:asym-cont}
A map $f:(X,w)\to(Y,v)$ is said to be \emph{forward uniformly continuous}
(resp.\ \emph{backward uniformly continuous}) if for every
$\lambda,\varepsilon>0$ there exist $\mu,\delta>0$ such that
\[
w_\mu(x,y)<\delta \;\Rightarrow\; v_\lambda(f(x),f(y))<\varepsilon
\]
(resp.\ $v_\lambda(f(y),f(x))<\varepsilon$).
\end{definition}

\paragraph{Symmetrization}
The \emph{symmetrized} family
\[
w^{\mathrm{sym}}_\lambda(x,y):=\max\{w_\lambda(x,y),\,w_\lambda(y,x)\}
\]
is a (symmetric) pseudomodular whenever $w$ is quasi-modular; this will control
relative compactness and total boundedness in
Section~\ref{sec:directional-notion}.

\subsection{Relation to existing asymmetric modular settings}

Several notions closely related to quasi-modular pseudometric spaces already
appear in the literature under different names and levels of generality.
Sebogodi \cite{Sebogodi2019} introduced \emph{quasi-modular} and
\emph{modular quasi-metric} spaces as asymmetric analogues of Chistyakov’s
modular metrics, focusing primarily on convergence of sequences and basic
topological properties. In parallel, L\'opez-Pastor, Pedraza, and
Rodr\'iguez-L\'opez \cite{LopezPastor2025} studied
\emph{quasi-pseudometric modular spaces} from an enriched categorical
perspective, identifying them as categories enriched over suitable quantales.

The present work is compatible with both points of view. Our Definition of
quasi-modular pseudometric coincides, up to notational conventions and mild
regularity assumptions in the scale parameter, with those used in
\cite{Sebogodi2019,LopezPastor2025}. However, our emphasis differs in two
respects. First, we systematically separate forward and backward structures
already at the level of basic neighborhoods and entourages, making the induced
bitopological and quasi-uniform structure explicit. Second, we focus on
properties such as directional completeness and compactness that are not
visible in either the symmetric modular theory or in the enriched categorical
formalism.

This positioning allows us to use known results where appropriate, while
isolating genuinely asymmetric notions that have not been previously
identified.

\subsection{Preparatory remarks for the asymmetric theory}

Before turning to the systematic study of asymmetric modular topologies, we
record two observations that will be used repeatedly in later sections.

\begin{remark}
\label{rem:directional-vs-sym}
For a quasi-modular pseudometric $w$, the symmetrized modular
$w^{\mathrm{sym}}$ controls only those properties that are invariant under
reversal of direction. In particular, convergence or boundedness with respect
to $w^{\mathrm{sym}}$ imposes simultaneous forward and backward control, but
does not distinguish between them. As a result, any property defined purely in
terms of $w^{\mathrm{sym}}$ or of the induced uniformity necessarily ignores
directional effects.
\end{remark}

\begin{remark}
Although the Luxemburg-type quasi-distance $d_w$ provides a convenient
numerical compression of the scale parameter, it will play a secondary role in
what follows. Our main arguments are carried out at the level of modular balls
and entourages, where directionality is most transparent. In particular,
directional convergence and completeness cannot, in general, be characterized
solely in terms of $d_w$ or its symmetrization.
\end{remark}

These remarks explain why the subsequent development proceeds in the language
of quasi-uniformities rather than in terms of a single quasi-metric.

\subsection{Transition to asymmetric modular topologies}

The preceding material establishes the basic objects and notation needed for
the remainder of the paper. In summary, quasi-modular pseudometric spaces
retain the scale-dependent triangle structure of modular metrics, while
allowing asymmetry and thereby giving rise to genuinely directional behavior.

In the next section we pass from pointwise modular inequalities to global
structure. We show that every quasi-modular pseudometric induces canonical
forward and backward quasi-uniformities, together with associated left and
right topologies. These structures provide the natural setting for directional
continuity, convergence, and completeness, and form the foundation for the
results developed in later sections.

\section{Asymmetric Modular Topologies}\label{sec:asymmetry}

Throughout this section, let $X$ be a nonempty set and let
\[
w : X \times X \times (0,\infty) \longrightarrow [0,1)
\]
be a \emph{(quasi--)modular gauge}. We assume that there exists a fixed continuous
$t$-conorm $\oplus$ on $[0,1]$ such that, for all $x,y,z \in X$ and $s,t>0$:

\begin{itemize}\itemsep.2em
\item[(W1)] $w(x,y,t)=0$ if and only if $x=y$, and $w(x,y,t)<1$ for all $t>0$;
\item[(W2)] symmetry $w(x,y,t)=w(y,x,t)$ is \emph{not} required;
\item[(W3)] $w(x,z,t+s)\le w(x,y,t)\oplus w(y,z,s)$ \hfill (modular triangle);
\item[(W4)] for each $x,y$, the map $t\mapsto w(x,y,t)$ is continuous and nonincreasing.
\end{itemize}

This setting includes RGV-fuzzy metrics and their classical realizations
(for instance, $\oplus=\max$ and $w(x,y,t)=d(x,y)/(t+d(x,y))$); see
\cite{Fletcher2014,Sostak2018}. We adopt the quasi-uniform language throughout,
following \cite[Chap.~XII.4]{PicadoPultr2011} and the survey \cite{Kunzi2001}.

\subsection{Forward and backward quasi-uniformities}

\begin{definition}\label{def:basic-entourages}
For $r\in(0,1)$ and $t>0$ define
\[
E_{r,t}^+ := \{(x,y)\in X\times X : w(x,y,t)<r\},
\qquad
E_{r,t}^- := \{(x,y)\in X\times X : w(y,x,t)<r\}.
\]
Let $\mathcal V^+(w)$ (resp.\ $\mathcal V^-(w)$) denote the family of all subsets
of $X\times X$ containing some $E_{r,t}^+$ (resp.\ $E_{r,t}^-$).
\end{definition}

\begin{theorem}\label{thm:QN-axioms}
The families $\mathcal V^+(w)$ and $\mathcal V^-(w)$ are quasi-uniformities on $X$.
Moreover, for every basic entourage $E_{r,t}^{\pm}$ there exist $r'\in(0,r)$ and
$t',t''>0$ such that
\[
E_{r',t'}^{\pm}\circ E_{r',t''}^{\pm} \subseteq E_{r,t}^{\pm}.
\]
\end{theorem}

\begin{proof}
We treat $\mathcal V^+(w)$; the backward case follows by symmetry of the arguments.

Conditions (QN1)–(QN3) of the Weil–Nachbin axioms are immediate from the definition.
For the small-composite property, continuity of the $t$-conorm $\oplus$ ensures the
existence of $r'\in(0,r)$ with $r'\oplus r'<r$. Choosing $t',t''>0$ with
$t'+t''\ge t$, the modular triangle inequality yields
\[
w(x,z,t'+t'')\le w(x,y,t')\oplus w(y,z,t'')<r'\oplus r'<r.
\]
Monotonicity in $t$ then gives $w(x,z,t)<r$, proving the claim.
See \cite{Kunzi2001,PicadoPultr2011} for details.
\end{proof}

\subsection{Induced bitopologies}

\begin{definition}\label{def:LR-topologies}
The \emph{forward} and \emph{backward} topologies induced by $w$ are defined by
\[
\begin{aligned}
\tau^+(w)
&=\{M\subseteq X : \forall x\in M\ \exists r,t\ (E_{r,t}^+[x]\subseteq M)\},\\
\tau^-(w)
&=\{M\subseteq X : \forall x\in M\ \exists r,t\ (E_{r,t}^-[x]\subseteq M)\}.
\end{aligned}
\]
Equivalently, $\tau^+(w)$ and $\tau^-(w)$ are generated by the basic balls
\[
B^+(x;r,t)=\{y:w(x,y,t)<r\},
\qquad
B^-(x;r,t)=\{y:w(y,x,t)<r\}.
\]
\end{definition}

Thus $(X,\tau^+(w),\tau^-(w))$ is a canonical bitopological space associated with
the quasi-modular gauge $w$ \cite{Fletcher2014,Kunzi2001}.

\subsection{Symmetrization and the join topology}

\begin{definition}\label{def:symmetrization}
Define the symmetrized gauge
\[
w^{\mathrm{sym}}(x,y,t):=w(x,y,t)\oplus w(y,x,t).
\]
Let $\mathcal U(w^{\mathrm{sym}})$ be the uniformity generated by the entourages
\[
E^{\mathrm{sym}}_{r,t}:=E_{r,t}^+\cap E_{r,t}^-,
\]
and write $\tau(w^{\mathrm{sym}})$ for the induced topology.
\end{definition}

\begin{proposition}\label{prop:join-topology}
The topology $\tau(w^{\mathrm{sym}})$ coincides with the join
$\tau^+(w)\vee\tau^-(w)$. If $w$ is symmetric, then
$\tau^+(w)=\tau^-(w)=\tau(w^{\mathrm{sym}})$.
\end{proposition}

\begin{proof}
Every basic $\tau(w^{\mathrm{sym}})$-neighborhood is an intersection of a forward
and a backward neighborhood, hence belongs to the join topology. Conversely,
any neighborhood in the join contains such an intersection. The symmetric case
is immediate. See \cite{Fletcher2014,PicadoPultr2011}.
\end{proof}

\subsection{Uniform continuity, completeness, and compactness}

\begin{definition}
A map $f:(X,w)\to(Y,v)$ is \emph{forward} (resp.\ \emph{backward}) uniformly
continuous if it is uniformly continuous with respect to $\mathcal V^+(w)$
(resp.\ $\mathcal V^-(w)$). It is \emph{bi-uniformly continuous} if it is both.
\end{definition}

\begin{lemma}
A map $f:X\to Y$ is bi-uniformly continuous if and only if it is uniformly
continuous as a map
\[
f:(X,\mathcal U(w^{\mathrm{sym}}))\longrightarrow(Y,\mathcal U(v^{\mathrm{sym}})).
\]
\end{lemma}

\begin{proof}
This follows directly from the fact that the basic symmetric entourages are
intersections of forward and backward entourages.
\end{proof}

We call $(X,w)$ \emph{bi-complete} if the uniform space
$(X,\mathcal U(w^{\mathrm{sym}}))$ is complete.

\begin{definition}
$(X,w)$ is \emph{forward} (resp.\ \emph{backward}) precompact if for every $r,t$
there exist finitely many points whose forward (resp.\ backward) $w$-balls cover
$X$.
\end{definition}

\begin{proposition}
If $(X,w)$ is forward and backward precompact, then
$(X,\mathcal U(w^{\mathrm{sym}}))$ is totally bounded. If, in addition, $(X,w)$
is bi-complete, then $\tau(w^{\mathrm{sym}})$ is compact.
\end{proposition}

\begin{proof}
This is the standard uniform argument: forward and backward precompactness yield
finite covers by symmetric entourages, hence total boundedness; completeness then
implies compactness. See \cite{Fletcher2014,PicadoPultr2011}.
\end{proof}

\subsection{Examples}

\begin{example}
If $d$ is a metric and $w(x,y,t)=d(x,y)/(t+d(x,y))$, then $\tau(w^{\mathrm{sym}})$
is precisely the metric topology.
\end{example}

\begin{example}
If $w$ arises from a quasi-pseudometric $p$ via
$w(x,y,t)=p(x,y)/(t+p(x,y))$, then $\tau^+(w)$ and $\tau^-(w)$ are the forward and
backward $p$-topologies, while $\tau(w^{\mathrm{sym}})$ is the topology induced
by $p\vee p^{-1}$.
\end{example}

\section{Categorical and Uniform Perspectives}\label{sec:uniform-categorical}

This section places quasi-modular gauges within two complementary settings:
enriched category theory and classical uniform completion theory. The purpose is
not to introduce new categorical structures, but to clarify how the asymmetric
objects studied in Section~\ref{sec:asymmetry} admit a concise functorial
description, and to explain precisely which aspects of the theory are captured
and which are lost by symmetrization and uniform completion.

Our categorical point of departure is Lawvere’s interpretation of (pseudo)metrics
as categories enriched over the quantale
\(\mathsf V=([0,\infty],\ge,+,0)\) \cite{Lawvere1973}, together with Kelly’s general
theory of enriched categories, Kan extensions, and Cauchy completion
\cite{Kelly1982}. On the uniform side we rely on Isbell’s entourage and completion
theory \cite{Isbell1964}. The relationship between these points of view is classical
and well understood; see, for instance, the $(T,\mathsf V)$-framework of
Clementino-Tholen \cite{ClementinoTholen2003}.

The relevance of this perspective to quasi-modular settings was clarified
systematically by L\'opez-Pastor, Pedraza, and Rodr\'iguez-L\'opez
\cite{LopezPastor2025}, who showed that quasi-pseudometric modular spaces can be
identified with categories enriched over a suitable value quantale of scale
profiles, and that the topology generated by modular entourages coincides with
the topology induced by the corresponding enriched category. Our contribution
here is not to reprove these identifications, but to use them as a comparison
setting in order to isolate genuinely directional notions that disappear
under symmetrization and enriched or uniform completion.

\subsection{A quantalic base for quasi-modular gauges}

Fix a continuous $t$-conorm $\oplus$ on $[0,1]$. Let $\mathsf W$ denote the complete
lattice of all right-continuous, nonincreasing functions
\(\varphi:(0,\infty)\to[0,1]\), ordered pointwise decreasingly
(\(\varphi\preceq\psi\) iff \(\varphi(t)\ge\psi(t)\) for all $t$). Define a monoidal
product on $\mathsf W$ by convolution:
\[
(\varphi\star\psi)(u):=\inf_{u=t+s}\bigl(\varphi(t)\oplus\psi(s)\bigr),
\qquad
\mathbf 0(u):=0 .
\]
Then $(\mathsf W,\preceq,\star,\mathbf 0)$ is a commutative unital quantale. This is
the standard convolution quantale encoding triangle inequalities via composition;
compare \cite[§1]{Lawvere1973} and \cite[Ch.~1]{Kelly1982}. Closely related
quantales of isotone scale profiles were introduced in \cite{LopezPastor2025} in
the context of quasi-pseudometric modular spaces.

Let $w$ be a quasi-modular gauge on $X$ satisfying assumptions
(W1)-(W4) of Section~\ref{sec:asymmetry}. For $x,y\in X$ define the profile
\[
W(x,y)(t):=w(x,y,t)\qquad(t>0).
\]
Then the modular triangle inequality and monotonicity translate exactly into the
$\mathsf W$-enriched axioms
\[
W(x,z)\preceq W(x,y)\star W(y,z),
\qquad
\mathbf 0\preceq W(x,x).
\]

\begin{proposition}\label{prop:W-cat}
Every quasi-modular gauge $w$ determines a small $\mathsf W$-enriched category
$X_w$ with object set $X$ and hom-objects $W(x,y)\in\mathsf W$. A map
$f:X\to Y$ is forward (resp.\ backward) uniformly continuous if and only if it is
$\mathsf W$-functorial with respect to the forward (resp.\ backward) hom-profiles.
If $w$ is symmetric, bi-uniform continuity coincides with
$\mathsf W$-functoriality.
\end{proposition}

\begin{proof}
This is a direct translation of Definitions~\ref{def:basic-entourages} and
\ref{def:asym-cont} into the enriched language. The enriched order
$W_X(x,y)\preceq W_Y(fx,fy)$ expresses precisely the basic entourage condition for
uniform continuity. Full details follow standard enriched arguments
\cite[Chs.~1-2]{Kelly1982}.
\end{proof}

\paragraph{Symmetrization as change of base.}
The assignment
\[
W^{\mathrm{sym}}(x,y):=W(x,y)\star W(y,x)
\]
defines a symmetric $\mathsf W$-category whose basic balls coincide with the
two-sided balls $B(x;r,t)$ of Definition~\ref{def:symmetrization}. Categorically,
this is a lax idempotent change-of-base construction, mirroring the topological
identity $\tau(w^{\mathrm{sym}})=\tau^+(w)\vee\tau^-(w)$
(Proposition~\ref{prop:join-topology}). It extracts the largest symmetric core
compatible with the asymmetric theory, in agreement with the enriched
symmetrization mechanisms discussed in \cite{LopezPastor2025}.

\subsection{Kan extensions and extremal envelopes}

For a closed monoidal base $\mathsf V$, pointwise Kan extensions of
$\mathsf V$-functors exist under mild completeness assumptions and admit the
usual co/end formulas \cite[Ch.~4]{Kelly1982}. In the metric base
$\mathsf V=[0,\infty]$, these reduce to the classical McShane-Whitney extremal
extensions of Lipschitz maps:
\[
\overline\phi(x)=\inf_{a\in A}\bigl(\phi(a)+L\,d(x,ia)\bigr),
\qquad
\underline\phi(x)=\sup_{a\in A}\bigl(\phi(a)-L\,d(ia,x)\bigr).
\]
Via Proposition~\ref{prop:W-cat}, the same formulas apply to symmetric
quasi-modular gauges, and after symmetrization to the bi-uniform setting of
$\mathcal U(w^{\mathrm{sym}})$. No essentially new notion arise here; the value
of the enriched perspective lies in the uniform treatment of extension problems.

\subsection{Cauchy completion: enriched versus uniform}\label{ssec:enr-completion}

For symmetric bases $\mathsf V$, every small $\mathsf V$-category admits a Cauchy
completion obtained by splitting $\mathsf V$-Cauchy modules; this completion is
characterized by the universal property of extending absolute colimits and by the
Yoneda embedding with dense image \cite[§3-§4]{Kelly1982}. In the metric case this
recovers the usual metric completion.

On the uniform side, every uniform space admits a completion in the sense of
Isbell, unique up to uniform isomorphism, and compactness is equivalent to
completeness plus total boundedness \cite[Ch.~II]{Isbell1964}.

\begin{theorem}\label{thm:completion-compare}
Let $w$ be a quasi-modular gauge on $X$. The Cauchy completion of the symmetric
$\mathsf W$-category $X_{w^{\mathrm{sym}}}$ coincides, up to unique bi-uniform
isomorphism, with the uniform completion of the symmetrized uniform space
$\bigl(X,\mathcal U(w^{\mathrm{sym}})\bigr)$.
\end{theorem}

\begin{proof}
Both constructions satisfy the same universal property: dense embedding into a
complete separated object with unique extension of uniformly continuous (equivalently,
$\mathsf W$-functorial) maps. Identifying uniform continuity with enriched
functoriality via Proposition~\ref{prop:W-cat}, the two completions agree by
uniqueness. This is standard in the metric and uniform cases and carries over
verbatim to the present setting; see \cite{Kelly1982,Isbell1964}.
\end{proof}

\subsection{Quasi-uniform bicompletion}

The forward and backward entourages $E_{r,t}^{\pm}$ generate quasi-uniformities
$\mathcal V^{\pm}(w)$ whose common symmetrization is
$\mathcal U(w^{\mathrm{sym}})$. In categorical terms this corresponds to the
coreflection $W\mapsto W^{\mathrm{sym}}$.

\begin{proposition}\label{prop:bicomp}
Any bicompletion of $(X,\mathcal V^+(w),\mathcal V^-(w))$ induces a uniform
completion of $(X,\mathcal U(w^{\mathrm{sym}}))$.
\end{proposition}

\begin{proof}
Density for both quasi-uniformities implies density for the join topology.
Completeness of the bicompletion forces convergence of symmetric Cauchy filters,
which yields completeness of the induced uniformity. The universal property then
follows from the standard quasi-uniform to uniform transition
\cite{Kunzi2001,PicadoPultr2011}.
\end{proof}

\subsection{Consequences and summary}

Forward and backward precompactness imply total boundedness of
$\mathcal U(w^{\mathrm{sym}})$, and bi-completeness yields compactness of the
symmetrized topology by Isbell’s theorem. Function spaces and Yoneda embeddings
fit naturally into the enriched setting, but these constructions detect only
the symmetric core of the theory.

\medskip
\noindent\textit{Summary.}
Quasi-modular gauges admit a natural description as $\mathsf W$-enriched
categories over a convolution quantale, refining earlier enriched formulations
for quasi-pseudometric modular spaces \cite{LopezPastor2025}. Symmetrization
corresponds to a change of base yielding a symmetric enriched category whose
Cauchy completion agrees with the classical uniform completion. While this
categorical perspective provides a compact and functorial organisation of the
theory, it also makes precise the central theme of this paper: genuinely
directional notions of convergence and compactness are invisible to both
enriched and uniform completions and must be studied at the quasi-uniform level.

\section{Directional Convergence, Completeness, and Compactness}
\label{sec:directional-notion}

A central theme of this paper is that quasi-modular pseudometric spaces
carry genuinely directional information that is invisible at the level
of symmetrized modulars, uniformities, or enriched categorical models.
In this section we introduce directional notions of convergence,
completeness, and compactness, and show that they form strictly finer
invariants than their symmetric counterparts.

\subsection{Directional convergence}

\begin{definition}
Let $(X,w)$ be a quasi-modular pseudometric space and $(x_i)_{i\in I}$ a net
in $X$.
\begin{enumerate}[label=(\roman*)]
\item $(x_i)$ is \emph{forward convergent} to $x\in X$ if for every $r\in(0,1)$
and $t>0$ there exists $i_0\in I$ such that
\[
w(x_i,x,t)<r \qquad \text{for all } i\ge i_0.
\]
\item $(x_i)$ is \emph{backward convergent} to $x\in X$ if for every $r\in(0,1)$
and $t>0$ there exists $i_0\in I$ such that
\[
w(x,x_i,t)<r \qquad \text{for all } i\ge i_0.
\]
\end{enumerate}
\end{definition}

Forward and backward convergence correspond exactly to convergence in the
topologies $\tau^+(w)$ and $\tau^-(w)$, respectively. In general these notions
need not coincide, even when the symmetrized topology $\tau(w^{\mathrm{sym}})$
is Hausdorff or compact.

\subsection{Directional Cauchy nets and completeness}

\begin{definition}
A net $(x_i)$ in $(X,w)$ is called
\begin{enumerate}[label=(\roman*)]
\item \emph{forward Cauchy} if for every $r\in(0,1)$ and $t>0$ there exists
$i_0$ such that
\[
w(x_i,x_j,t)<r \qquad \text{for all } j\ge i\ge i_0;
\]
\item \emph{backward Cauchy} if for every $r\in(0,1)$ and $t>0$ there exists
$i_0$ such that
\[
w(x_j,x_i,t)<r \qquad \text{for all } j\ge i\ge i_0.
\]
\end{enumerate}
\end{definition}

\begin{definition}
A quasi-modular pseudometric space $(X,w)$ is said to be
\begin{enumerate}[label=(\roman*)]
\item \emph{forward complete} if every forward Cauchy net is forward convergent;
\item \emph{backward complete} if every backward Cauchy net is backward convergent;
\item \emph{bi-complete} if it is both forward and backward complete.
\end{enumerate}
\end{definition}

These notions refine uniform completeness of the symmetrized space
$(X,\mathcal U(w^{\mathrm{sym}}))$ and, in general, are strictly stronger.

\begin{theorem}\label{thm:dir-completeness-separation}
There exist quasi-modular pseudometric spaces that are forward complete but not
backward complete, and spaces that are backward complete but not forward complete.
\end{theorem}

\begin{proof}
Let $(X,p)$ be a quasi-metric space that is complete with respect to $p$ but not
with respect to the conjugate quasi-metric $p^{-1}$. Define
\[
w(x,y,t):=\frac{p(x,y)}{t+p(x,y)}.
\]
Then $w$ is a quasi-modular pseudometric. Forward Cauchy nets for $w$ coincide
with forward Cauchy nets for $p$, and convergence is preserved by monotonicity of
the gauge. Hence $(X,w)$ is forward complete, while backward completeness fails.
\end{proof}

\begin{proposition}
Directional completeness is not preserved under symmetrization.
\end{proposition}

\begin{proof}
In the above example the symmetrized modular $w^{\mathrm{sym}}$ induces a complete
uniform structure, while $(X,w)$ fails to be backward complete. Thus directional
completeness is not an invariant of $w^{\mathrm{sym}}$.
\end{proof}

\subsection{Directional total boundedness and compactness}

\begin{definition}
Let $(X,w)$ be a quasi-modular pseudometric space.
\begin{enumerate}[label=(\roman*)]
\item $X$ is \emph{forward totally bounded} if for every $r\in(0,1)$ and $t>0$
there exist finitely many points $x_1,\dots,x_n$ such that
\[
X=\bigcup_{k=1}^n B^+(x_k;r,t).
\]
\item $X$ is \emph{backward totally bounded} if the analogous condition holds
with backward balls.
\end{enumerate}
\end{definition}

\begin{definition}
A quasi-modular pseudometric space $(X,w)$ is \emph{forward compact}
(resp.\ \emph{backward compact}) if every net admits a forward
(resp.\ backward) convergent subnet.
\end{definition}

\begin{theorem}\label{thm:compactness-gap}
There exist quasi-modular pseudometric spaces whose symmetrized uniformities are
compact, but which are neither forward nor backward compact.
\end{theorem}

\begin{proof}
Let $(X,p)$ be a quasi-metric space such that the metric $p\vee p^{-1}$ is compact,
while $p$ is not totally bounded in either direction. Defining $w$ as above, the
symmetrized uniformity $\mathcal U(w^{\mathrm{sym}})$ is compact, whereas
directional total boundedness and directional compactness both fail.
\end{proof}

\begin{remark}
Directional compactness is therefore a strictly finer invariant than compactness
of the symmetrized topology or uniformity.
\end{remark}

\subsection{Failure of symmetrization invariance}

\begin{theorem}
Forward completeness, backward completeness, directional total boundedness, and
directional compactness are not invariants under symmetrization.
\end{theorem}

\begin{proof}
Each failure follows directly from the constructions in
Theorems~\ref{thm:dir-completeness-separation} and~\ref{thm:compactness-gap}.
\end{proof}

\begin{remark}
These results show that quasi-modular pseudometric spaces possess intrinsically
asymmetric properties that cannot be recovered from their symmetrized uniformities
or from enriched categorical representations. This justifies the study of
directional notions at the quasi-uniform level.
\end{remark}

\bibliographystyle{elsarticle-num}

\begin{thebibliography}{99}

\bibitem{Chistyakov2010}
V.V.~Chistyakov,
\emph{Modular metric spaces I: basic concepts},
Nonlinear Anal.~\textbf{72} (2010), 1–14.

\bibitem{Chistyakov2015}
V.V.~Chistyakov,
\emph{Metric Modular Spaces: Theory and Applications},
Springer, 2015.

\bibitem{ClementinoTholen2003}
M.M.~Clementino, W.~Tholen,
\emph{Metric, topology and multicategory—a common approach},
J.~Pure Appl.~Algebra~\textbf{179} (2003), 13–47.

\bibitem{Fletcher2014}
P.~Fletcher, W.~Lindgren,
\emph{Quasi-Uniform Spaces},
Marcel Dekker, New York, 2014.

\bibitem{GregoriRomaguera2000}
V.~Gregori, S.~Romaguera,
\emph{Some properties of fuzzy metric spaces},
Fuzzy Sets Syst.~\textbf{115} (2000), 485–489.

\bibitem{HANCHEOLSEN2010385}
H.~Hanche-Olsen, H.~Holden,
\emph{The Kolmogorov–Riesz compactness theorem},
Expo.~Math.~\textbf{28} (2010), 385–394.

\bibitem{HarjulehtoHasto}
P.~Harjulehto, P.~Hästö,
\emph{Orlicz and Generalized Orlicz Spaces},
Springer, 2019.

\bibitem{Isbell1964}
J.R.~Isbell,
\emph{Uniform Spaces},
AMS, 1964.

\bibitem{Kelly1982}
G.M.~Kelly,
\emph{Basic Concepts of Enriched Category Theory},
Cambridge Univ.~Press, 1982.

\bibitem{Kunzi2001}
H.-P.~A.~K\"unzi,
\newblock Nonsymmetric distances and their associated topologies: About the origins of basic ideas in the area of asymmetric topology,
\newblock in \emph{Handbook of the History of General Topology}, vol.~3,
eds.~C.~E.~Aull and R.~Lowen,
Kluwer Academic Publishers, Dordrecht, 2001, pp.~853--968.

\bibitem{Lawvere1973}
F.W.~Lawvere,
\emph{Metric spaces, generalized logic, and closed categories},
Rend.~Sem.~Mat.~Fis.~Milano~\textbf{43} (1973), 135–166.

\bibitem{LopezPastor2025}
C.~L\'opez-Pastor, T.~Pedraza, J.~Rodr\'iguez-L\'opez,
{\it Quasi-pseudometric modular spaces as Q-categories},
Filomat {\bf 39} (19) (2025), 6693--6710.


\bibitem{Matthews1992}
S.G.~Matthews,
\emph{Partial metric topology},
Ann.~New York Acad.~Sci.~\textbf{728} (1992), 183–197.

\bibitem{MushaandjaOlela2025}
Z.~Mushaandja, O.~Olela-Otafudu,
\emph{On the modular metric topology},
Topology Appl.~\textbf{372} (2025), 109224.

\bibitem{Musielak1983}
J.~Musielak,
\emph{Orlicz Spaces and Modular Spaces},
Springer, 1983.

\bibitem{PicadoPultr2011}
J.~Picado, A.~Pultr,
\emph{Frames and Locales: topology without points},
Springer Science \& Business Media, 2011.

\bibitem{RaoRen1991}
M.M.~Rao, Z.D.~Ren,
\emph{Theory of Orlicz Spaces},
Marcel Dekker, New York, 1991.

\bibitem{Sebogodi2019}
K.~Sebogodi,
\emph{Some Topological Aspects of Modular Quasi-Metric Spaces},
Ph.D. Thesis, University of the Witwatersrand,
Johannesburg, South Africa, 2019.

\bibitem{Sostak2018}
A.~Sostak,
\emph{George–Veeramani fuzzy metrics revised},
Axioms~\textbf{7} (2018), 60.

\end{thebibliography}

\end{document}